\documentclass[a4,12pt]{article}
\usepackage{amssymb,latexsym,amsmath,amsthm,subfig,graphicx,anysize,color,lineno,fourier}
\usepackage[affil-it]{authblk}
\marginsize{2cm}{2cm}{2cm}{2cm}

\newcommand{\eps}{\varepsilon}
\newcommand{\set}[1]{\left\{#1\right\}}

\newcommand{\p}{\partial}
\newcommand{\mc}{\mathbf{c}}

\newcommand{\me}{\mathbf{e}}
\newcommand{\mf}{\mathbf{f}}

\newcommand{\mr}{\mathbf{r}}

\newcommand{\mx}{\mathbf{x}}
\newcommand{\my}{\mathbf{y}}
\newcommand{\mz}{\mathbf{z}}
\newcommand{\mB}{\mathbf{B}}
\newcommand{\mE}{\mathbf{E}}
\newcommand{\mG}{\mathbf{G}}
\newcommand{\mH}{\mathbf{H}}
\newcommand{\mU}{\mathbf{U}}

\newcommand{\mW}{\mathbf{W}}
\newcommand{\vt}{\boldsymbol{\theta}}
\newcommand{\vv}{\boldsymbol{\vartheta}}
\newcommand{\vx}{\boldsymbol{\xi}}

\newtheorem{thm}{Theorem}[section]

\newtheorem{lem}[thm]{Lemma}

\newtheorem{Rem}[thm]{Remark}

\begin{document}

\title{Interpretation of MUSIC for location detecting of small inhomogeneities surrounded by random scatterers}
\author{Won-Kwang Park}
\affil{Department of Mathematics, Kookmin University, Seoul, 02707, Republic of Korea.\\
e-mail: parkwk@kookmin.ac.kr}
\date{}
\maketitle

\begin{abstract}
  In this paper, we consider the MUltiple SIgnal Classification (MUSIC) algorithm for identifying the locations of small electromagnetic inhomogeneities surrounded by random scatterers. For this purpose, we rigorously analyze the structure of MUSIC-type imaging function by establishing a relationship with zero-order Bessel function of the first kind. This relationship shows certain properties of the MUSIC algorithm, explains some unexplained phenomena, and provides a method for improvements.
\end{abstract}

\section{Introduction}
One of the purposes of the inverse scattering problem is to identify the characteristics (location, shape, material properties, etc.) of small inhomogeneities from the scattered field or far-field pattern. This problem, which arises in fields such as physics, engineering, and biomedical science, is highly relevant to human life; thus, it remains an important research area. Related works can be found in \cite{A3,ABF,FKM,KDAK,SKLKLJC} and references therein.

Attempts to address the problem described above have led to the development of the MUltiple SIgnal Classification (MUSIC)-type algorithm to find unknown inhomogeneities and the algorithm has been applied to various problems, e.g., detection of small inhomogeneities in homogeneous space \cite{AIL2,AILP,IGLP,ZC}, location identification of small inhomogeneities embedded in a half-space or multi-layered medium \cite{AIL1,G2,SCC}, reconstructing perfectly conducting cracks \cite{AGKPS,AKLP}, imaging of internal corrosion \cite{AKKLL}, shape recognition of crack-like thin inhomogeneities \cite{AJP,P-MUSIC1,PL3} and volumetric extended targets \cite{AGKLS,HSZ1,HSZ2}, and application to the biomedical imaging \cite{S2}. We also refer to \cite{AK2,C} for a detailed and concise description of MUSIC. Several research efforts have contributed to confirming that MUSIC is a fast and stable algorithm that can easily be extended to multiple inhomogeneities, and that does not require specific regularization terms that are highly dependent on the problem at hand. However, its feasibility is only confirmed when the background medium is homogeneous, i.e., the imaging performance of MUSIC when unknown inhomogeneities are surrounded by random scatterers remains unknown. In several works \cite{CSDPS,BPTB,K1,S3}, an inverse scattering problem in random media has been concerned. Specially, mathematical theory of MUSIC for detecting point-like scatterers embedded in an inhomogeneous medium has been concerned in \cite{C1}. Motivated these remarkable works, a more careful investigation of the mathematical theory still required.

Motivated by the above, MUSIC algorithm has been applied for detecting the locations of small electromagnetic inhomogeneities when they are surrounded by electromagnetic random scatterers and confirmed that it can be applied satisfactorily. However, this only relied on the results of numerical simulations, i.e., a heuristic approach to some extent, which is the motivation for the current work. In this contribution, we carefully analyze the mathematical structure of a MUSIC-type imaging function and discover some properties. This work is based on the relationship between the singular vectors associated with nonzero singular values of a multi-static response (MSR) matrix and asymptotic expansion formula due to the existence of small inhomogeneities, refer to \cite{AK2}.

This paper is organized as follows. Section \ref{sec:2} introduces the two-dimensional direct scattering problem and an asymptotic expansion formula in the presence of small inhomogeneities. In Section \ref{sec:3}, a MUSIC-type imaging function is introduced. In Section \ref{sec:4}, we analyze the mathematical structure of the MUSIC-type imaging function and discuss its properties. In Section \ref{sec:5}, we present the results of numerical simulations to support the analyzed structure of MUSIC and Section \ref{sec:6} presents a short conclusion.

\section{Two-dimensional direct scattering problem}\label{sec:2}
In this section, we survey a two-dimensional direct scattering problem and introduce an asymptotic expansion formula. For a more detailed description we recommend \cite{PL3,AK2,RC}. Let $\Sigma_m$, $m=1,2,\cdots,M$, be an electromagnetic inhomogeneity with a small diameter $r_m$ in two-dimensional space $\mathbb{R}^2$. Throughout this paper, we assume that every $\Sigma_m$ is expressed as
\[\Sigma_m=\mz_m+r_m\mB_m,\]
where $\mz_m$ denotes the location of $\Sigma_m$ and $\mB_m$ is a simple connected smooth domain containing the origin. For the sake of simplicity, we let $\Sigma$ be the collection of $\Sigma_m$. Throughout this paper, we assume that inhomogeneities are well separated from each other such that
\begin{equation}\label{Separated}
  \omega|\mz_m-\mz_{m'}|\gg1-\frac{1}{4}=0.75,
\end{equation}
for all $m,m'=1,2,\cdots,M$ and $m\ne m'$.

Let us denote $\Delta_s$, $s=1,2,\cdots,S$, as the random scatterer with small radius $r_s<r$ and let $\Delta$ be the collection of $\Delta_s$. Similarly, we assume that $\Delta_s$ is of the form:
\[\Delta_s=\my_s+r_s\mB_m.\]
As before, suppose that $\Delta_s\cap\Delta_{s'}=\emptyset$ for all $s,s'=1,2,\cdots,S$ and $s\ne s'$ and the positions of $\my_s$ are random but they are fixed for all frequencies discussed later.

In this work, we assume that every inhomogeneity is characterized by its dielectric permittivity and magnetic permeability at a given positive angular frequency $\omega=2\pi/\lambda$, where $\lambda$ denotes the wavelength. Let $\eps_m$, $\eps_s$, and $\eps_0$ be the electric permittivities of $\Sigma_m$, $\Delta_s$, and $\mathbb{R}^2$, respectively. Then, we can introduce the piecewise-constant electric permittivity $\eps(\mx)$ and magnetic permeability $\mu(\mx)$ such that
\[\eps(\mx)=\left\{\begin{array}{ccl}
\medskip\eps_m&\mbox{for}&\mx\in\Sigma_m\\
\medskip\eps_s&\mbox{for}&\mx\in\Delta_s\\
\eps_0&\mbox{for}&\mx\in\mathbb{R}^2\backslash(\overline{\Sigma}\cup\overline{\Delta})
\end{array}\right.\]
and
\[\mu(\mx)=\left\{\begin{array}{ccl}
\medskip\mu_m&\mbox{for}&\mx\in\Sigma_m\\
\medskip\mu_s&\mbox{for}&\mx\in\Delta_s\\
\mu_0&\mbox{for}&\mx\in\mathbb{R}^2\backslash(\overline{\Sigma}\cup\overline{\Delta}),
\end{array}\right.\]
respectively. For the sake of simplicity, we let $\eps_0=\mu_0=1$, $\eps_m>\eps_s$, and $\mu_m>\mu_s$ for all $m$ and $s$. Hence, we can set the wavenumber $k=\omega\sqrt{\eps_0\mu_0}=\omega$.

For a given fixed frequency $\omega$, we denote
\[u_{\mathrm{inc}}(\mx,\vt)=e^{i\omega\vt\cdot\mx}\]
to be a plane-wave incident field with the incident direction $\vt\in\mathbb{S}^1$, where $\mathbb{S}^1$ denotes the two-dimensional unit circle. Let $u(\mx,\vt)$ denote the time-harmonic total field that satisfies the following Helmholtz equation
\[\nabla\cdot\left(\frac{1}{\mu(\mx)}\nabla u(\mx,\vt)\right)+\omega^2\eps(\mx)u(\mx,\vt)=0\]
with transmission conditions on the boundaries of $\Sigma_m$ and $\Delta_s$. This configuration is associated with a scalar scattering problem for an $E-$polarized (Transverse Magnetic-TM-polarization / corresponding to dielectric contrasts) field--the $H-$polarized (Transverse Electric-TE-polarization / corresponding to magnetic contrasts) case could be dealt with per duality. It is well known that $u(\mx,\vt)$ can be decomposed as
\[u(\mx,\vt)=u_{\mathrm{inc}}(\mx,\vt)+u_{\mathrm{scat}}(\mx,\vt),\]
where $u_{\mathrm{scat}}(\mx,\vt)$ denotes the unknown scattered field that satisfies the Sommerfeld radiation condition
\[\lim_{|\mx|\to0}\sqrt{|\mx|}\left(\frac{\p u_{\mathrm{scat}}(\mx,\vt)}{\p|\mx|}-i\omega u_{\mathrm{scat}}(\mx,\vt)\right)=0\]
uniformly in all directions $\vv=\frac{\mx}{|\mx|}\in\mathbb{S}^1$. The far-field pattern $u_{\mathrm{far}}(\vv,\vt)$ of the scattered field $u_{\mathrm{scat}}(\mx,\vt)$ is defined on $\mathbb{S}^1$. It can be expressed as
\[u_{\mathrm{scat}}(\mx,\vt)=\frac{e^{i\omega|\mx|}}{\sqrt{|\mx|}}u_{\mathrm{far}}(\vv,\vt)+o\left(\frac{1}{\sqrt{|\mx|}}\right),\quad|\mx|\longrightarrow+\infty.\]
Then by virtue of \cite{BF}, the far-field pattern $u_{\mathrm{far}}(\vv,\vt)$ can be written as the following asymptotic expansion formula, which plays a key role in the MUSIC-type algorithm that will be designed in the next section.
\begin{multline}\label{AsymptoticExpansionFormula}
  u_{\mathrm{far}}(\vv,\vt)=\frac{\omega^2(1+i)}{4\sqrt{\omega\pi}}\left\{\sum_{m=1}^{M}r_m^2\bigg((\eps_m-\eps_0)|\mB_m|-\frac{\mu_0}{\mu_m+\mu_0}(\sqrt2\vv)\cdot(\sqrt2\vt)\bigg)e^{i\omega(\vt-\vv)\cdot\mz_m}\right.\\
  \left.+\sum_{s=1}^{S}r_s^2\left(\frac{\eps_s-\eps_0}{\sqrt{\eps_0\mu_0}}|\mB_s|-\frac{\mu_0}{\mu_s+\mu_0}(\sqrt2\vv)\cdot(\sqrt2\vt)\right)e^{i\omega(\vt-\vv)\cdot\my_s}\right\}.
\end{multline}

\section{MUSIC-type imaging algorithm}\label{sec:3}
In this section, we introduce the MUSIC-type algorithm for detecting the locations of small inhomogeneities. For the sake of simplicity, we exclude the constant term $\frac{\omega^2(1+i)}{4\sqrt{\omega\pi}}$ from (\ref{AsymptoticExpansionFormula}). For this, let us consider  the eigenvalue structure of the MSR matrix
\[\mathbb{K}=\left[\begin{array}{cccc}
u_{\mathrm{far}}(\vv_1,\vt_1) & u_{\mathrm{far}}(\vv_1,\vt_2) & \cdots & u_{\mathrm{far}}(\vv_1,\vt_N)\\
u_{\mathrm{far}}(\vv_2,\vt_1) & u_{\mathrm{far}}(\vv_2,\vt_2) & \cdots & u_{\mathrm{far}}(\vv_2,\vt_N)\\
\vdots&\vdots&\ddots&\vdots\\
u_{\mathrm{far}}(\vv_N,\vt_1) & u_{\mathrm{far}}(\vv_N,\vt_2) & \cdots & u_{\mathrm{far}}(\vv_N,\vt_N)\\
\end{array}\right].\]
Suppose that $\vv_j=-\vt_j$ for all $j$, then $\mathbb{K}$ is a complex symmetric matrix but not a Hermitian. Thus, instead of eigenvalue decomposition, we perform singular value decomposition (SVD) of $\mathbb{K}$ (see \cite{C} for instance)
\begin{equation}\label{SVD}
  \mathbb{K}\approx\sum_{m=1}^{3M}\sigma_m\mathbf{U}_m\mathbf{V}_m^*+\sum_{s=3M+1}^{3M+3S}\sigma_s\mathbf{U}_s\mathbf{V}_s^*,
\end{equation}
where superscript $*$ is the mark of a Hermitian. Then, $\set{\mathbf{U}_1,\mathbf{U}_2,\cdots,\mathbf{U}_{3M+3S}}$ is the orthogonal basis for the signal space of $\mathbb{K}$. Therefore, one can define the projection operator onto the null (or noise) subspace, $\mathbf{P}_{\mathrm{noise}}:\mathbb{C}^{N\times1}\longrightarrow\mathbb{C}^{N\times1}$. This projection is given explicitly by
\begin{equation}
\mathbf{P}_{\mathrm{noise}}:=\mathbb{I}_N-\sum_{m=1}^{3M+3S}\mathbf{U}_m\mathbf{U}_m^*,
\end{equation}
where $\mathbb{I}_N$ denotes the $N\times N$ identity matrix. For any point $\mx\in\mathbb{R}^2$ and suitable vectors $\mc_n\in\mathbb{R}^3\backslash\set{\mathbf{0}}$, $n=1,2,\cdots,N$, define a test vector $\mf(\mx)\in\mathbb{C}^{N\times1}$ as
\begin{equation}\label{TestVector}
  \mf(\mx)=\bigg[\mc_1\cdot[1,\vt_1]^Te^{i\omega\vt_1\cdot\mx},\mc_2\cdot[1,\vt_2]^Te^{i\omega\vt_2\cdot\mx},\cdots,\mc_N\cdot[1,\vt_N]^Te^{i\omega\vt_N\cdot\mx}\bigg]^T.
\end{equation}
Then, by virtue of \cite{AK2}, there exists $N_0\in\mathbb{N}$ such that for any $N\geq N_0$, the following statement holds:
\[\mf(\mx)\in\mathrm{Range}(\mathbb{K\overline{\mathbb{K}}})\quad\mbox{if and only if}\quad\mx\in\set{\mz_m,\my_s}\]
for $m=1,2,\cdots,M$ and $s=1,2,\cdots,S$. This means that if $\mx\in\Sigma_m$ or $\mx\in\Delta_s$ then, $|\mathbf{P}_{\mathrm{noise}}(\mf(\mx))|=0$. Thus, the locations of $\Sigma_m$ and $\Delta_s$ follow from computing the MUSIC-type imaging function
\begin{equation}\label{MUSICimaging}
  \mathcal{F}(\mx)=\frac{1}{|\mathbf{P}_{\mathrm{noise}}(\mf(\mx))|}.
\end{equation}
The resulting plot of $\mathcal{F}(\mx)$ will have peaks of large magnitudes at $\mz_m\in\Sigma_m$ and $\my_s\in\Delta_s$.

\begin{Rem}
  Based on several works \cite{P-MUSIC1,PL3,HSZ1}, selection of $\mc_n$ in (\ref{TestVector}) is highly depending on the shape of $\Sigma_m$. Unfortunately, the shape of $\Sigma_m$ is unknown, it is impossible to find proper vectors $\mc_n$. Due to this fact, following from \cite{HSZ1}, we assume that $\mc_n\cdot[1,\vt_n]^T=1$ for all $n$, i.e., we consider the following test vector instead of (\ref{TestVector})
  \[\mf(\mx)=\frac{1}{\sqrt{N}}\bigg[e^{i\omega\vt_1\cdot\mx},e^{i\omega\vt_2\cdot\mx},\cdots,e^{i\omega\vt_N\cdot\mx}\bigg]^T\]
  and analyze the mathematical structure of $\mathcal{F}(\mx)$.
\end{Rem}

\section{Structure of imaging function}\label{sec:4}
Henceforth, we analyze the mathematical structure of $\mathcal{F}(\mx)$ and examine certain of its properties. Before starting, we recall a useful result derived in \cite{P-SUB3}.
\begin{lem}
Assume that $\set{\vt_n:n=1,2,\cdots,N}$ spans $\mathbb{S}^1$. Then, for sufficiently large $N$, $\vx\in\mathbb{S}^1$, and $\mx\in\mathbb{R}^2$, the following relation holds:
\begin{align}
\begin{aligned}\label{BesselRelation}
  &\frac{1}{N}\sum_{n=1}^{N}e^{i\omega\vt_n\cdot\mx}=\frac{1}{2\pi}\int_{\mathbb{S}^1}e^{i\omega\vt\cdot\mx}dS(\vt)=J_0(\omega|\mx|),\\
  &\frac{1}{N}\sum_{n=1}^{N}(\vx\cdot\vt_n)e^{i\omega\vt_n\cdot\mx}=\frac{1}{2\pi}\int_{\mathbb{S}^1}(\vx\cdot\vt)e^{i\omega\vt\cdot\mx}dS(\vt)=i\left(\frac{\mx}{|\mx|}\cdot\vx\right)J_1(\omega|\mx|),
\end{aligned}
\end{align}
where $J_n$ denotes Bessel function of order $n$ of the first kind.
\end{lem}

Now, we introduce the main result.
\begin{thm}\label{TheoremMUSIC}
For sufficiently large $N>3M+3S$ and $\omega$, $\mathcal{F}(\mx)$ can be represented as follows: for $\me_1=[1,0]^T$ and $\me_2=[0,1]^T$,
\begin{multline*}
  \mathcal{F}(\mx)\approx\left(1-\sum_{m=1}^{M}J_0(\omega|\mx-\mz_m|)^2-\sum_{m=1}^{M}\sum_{h=1}^2\bigg(\frac{(\mx-\mz_m)\cdot\me_h}{|\mx-\mz_m|}\bigg)^2J_1(\omega|\mx-\mz_m|)^2\right.\\
  \left.-\sum_{s=1}^{S}J_0(\omega|\mx-\my_s|)^2-\sum_{s=1}^{S}\sum_{h=1}^2\bigg(\frac{(\mx-\my_s)\cdot\me_h}{|\mx-\my_s|}\bigg)^2J_1(\omega|\mx-\my_s|)^2\right)^{-1/2}.
\end{multline*}
\end{thm}
\begin{proof}
Based on the asymptotic expansion formula (\ref{AsymptoticExpansionFormula}) and results in \cite{AGKPS}, $\mathbf{P}_{\mathrm{noise}}$ can be represented as
  \[\mathbf{P}_{\mathrm{noise}}=\mathbb{I}_N-\sum_{m=1}^{3M}\mU_m\mU_m^*-\sum_{s=1}^{3S}\mU_s\mU_s^*\approx\mathbb{I}_N-\sum_{m=1}^{M}\sum_{h=1}^{3}\mW_h(\mz_m)\mW_h(\mz_m)^*-\sum_{s=1}^{S}\sum_{h=1}^{3}\mW_h(\my_s)\mW_h(\my_s)^*,\]
  where
  \begin{align*}
    \mW_1(\mx)&=\frac{1}{\sqrt{N}}\bigg[e^{i\omega\vt_1\cdot\mx},e^{i\omega\vt_2\cdot\mx}\cdots,e^{i\omega\vt_N\cdot\mx}\bigg]^T,\\
    \mW_2(\mx)&=\frac{\sqrt{2}}{\sqrt{N}}\bigg[(\me_1\cdot\vt_1)e^{i\omega\vt_1\cdot\mx},(\me_1\cdot\vt_2)e^{i\omega\vt_2\cdot\mx},\cdots,(\me_1\cdot\vt_N)e^{i\omega\vt_N\cdot\mx}\bigg]^T,\\
    \mW_3(\mx)&=\frac{\sqrt{2}}{\sqrt{N}}\bigg[(\me_2\cdot\vt_1)e^{i\omega\vt_1\cdot\mx},(\me_2\cdot\vt_2)e^{i\omega\vt_2\cdot\mx},\cdots,(\me_2\cdot\vt_N)e^{i\omega\vt_N\cdot\mx}\bigg]^T.
  \end{align*}
  With this, applying (\ref{BesselRelation}) and performing a tedious calculation, we arrive at
  \[\mathbf{P}_{\mathrm{noise}}(\mf(\mx))=\mf(\mx)-\frac{1}{N\sqrt{N}}\left(\sum_{m=1}^{M}(\mathbb{A}(\mz_m)+\mathbb{B}_1(\mz_m)+\mathbb{B}_2(\mz_m))-\sum_{s=1}^{S}(\mathbb{A}(\my_s)+\mathbb{B}_1(\my_s)+\mathbb{B}_2(\my_s))\right),\]
  where
  \[\mathbb{A}(\vx):=\left[\begin{array}{c}
    \medskip e^{i\omega\vt_1\cdot\vx}J_0(\omega|\mx-\vx|) \\
    \medskip e^{i\omega\vt_2\cdot\vx}J_0(\omega|\mx-\vx|) \\
    \medskip\vdots \\
    e^{i\omega\vt_N\cdot\vx}J_0(\omega|\mx-\vx|)
    \end{array}\right],\]
  and
  \[\mathbb{B}_h(\vx):=\left[\begin{array}{c}
    \medskip\displaystyle i(\me_h\cdot\vt_1)\left(\frac{(\mx-\vx)\cdot\me_h}{|\mx-\vx|}\right)e^{i\omega\vt_1\cdot\vx}J_1(\omega|\mx-\vx|)\\
    \medskip\displaystyle i(\me_h\cdot\vt_2)\left(\frac{(\mx-\vx)\cdot\me_h}{|\mx-\vx|}\right)e^{i\omega\vt_2\cdot\vx}J_1(\omega|\mx-\vx|)\\
    \medskip\displaystyle\vdots \\
    \medskip\displaystyle i(\me_h\cdot\vt_N)\left(\frac{(\mx-\vx)\cdot\me_h}{|\mx-\vx|}\right)e^{i\omega\vt_N\cdot\vx}J_1(\omega|\mx-\vx|)
    \end{array}\right]\]
  for $\vx\in\mathbb{R}^2$ and $h=1,2$. By implementing elementary calculus, we can show that
  \[|\mathbf{P}_{\mathrm{noise}}(\mf(\mx))|^2=\mathbf{P}_{\mathrm{noise}}(\mf(\mx))\overline{\mathbf{P}_{\mathrm{noise}}(\mf(\mx))}
    =\frac{1}{N}\sum_{n=1}^{N}\bigg(1-\sum_{h=1}^{8}\Phi_h\bigg),\]
  where
  \begin{align*}
    \Phi_1=&\sum_{m=1}^{M}\bigg(e^{i\omega\vt_n\cdot(\mx-\mz_m)}+e^{-i\omega\vt_n\cdot(\mx-\mz_m)}\bigg)J_0(\omega|\mx-\mz_m|),\\
    \Phi_2=&-\left(\sum_{m=1}^{M}e^{i\omega\vt_n\cdot\mz_m}J_0(\omega|\mx-\mz_m|)\right)\left(\sum_{m'=1}^{M}e^{-i\omega\vt_n\cdot\mz_{m'}}J_0(\omega|\mx-\mz_{m'}|)\right),\\
    \Phi_3=&-i\sum_{m=1}^{M}\sum_{h=1}^{2}(\me_h\cdot\vt_n)\left(\frac{(\mx-\mz_m)\cdot\me_h}{|\mx-\mz_m|}\right)\bigg(e^{i\omega\vt_n\cdot(\mx-\mz_m)}-e^{-i\omega\vt_n\cdot(\mx-\mz_m)}\bigg)J_1(\omega|\mx-\mz_m|)\\
    \Phi_4=&-\left(\sum_{m=1}^{M}\sum_{h=1}^{2}(\me_h\cdot\vt_n)\left(\frac{(\mx-\mz_m)\cdot\me_h}{|\mx-\mz_m|}\right)e^{i\omega\vt_n\cdot\mz_m}J_1(\omega|\mx-\mz_m|)\right)\\
    &\times\left(\sum_{m'=1}^{M}\sum_{h=1}^{2}(\me_h\cdot\vt_n)\left(\frac{(\mx-\mz_{m'})\cdot\me_h}{|\mx-\mz_{m'}|}\right)e^{i\omega\vt_n\cdot\mz_{m'}}J_1(\omega|\mx-\mz_{m'}|)\right),\\
    \Phi_5=&\sum_{s=1}^{S}\bigg(e^{i\omega\vt_n\cdot(\mx-\my_s)}+e^{-i\omega\vt_n\cdot(\mx-\my_s)}\bigg)J_0(\omega|\mx-\my_s|),\\
    \Phi_6=&-\left(\sum_{s=1}^{S}e^{i\omega\vt_n\cdot\my_s}J_0(\omega|\mx-\my_s|)\right)\left(\sum_{s'=1}^{S}e^{-i\omega\vt_n\cdot\my_{s'}}J_0(\omega|\mx-\my_{s'}|)\right),\\
    \Phi_7=&-i\sum_{s=1}^{M}\sum_{h=1}^{2}(\me_h\cdot\vt_n)\left(\frac{(\mx-\my_s)\cdot\me_h}{|\mx-\my_s|}\right)\bigg(e^{i\omega\vt_n\cdot(\mx-\my_s)}+e^{-i\omega\vt_n\cdot(\mx-\my_s)}\bigg)J_1(\omega|\mx-\my_s|)\\
    \Phi_8=&-\left(\sum_{s=1}^{M}\sum_{h=1}^{2}(\me_h\cdot\vt_n)\left(\frac{(\mx-\my_s)\cdot\me_h}{|\mx-\my_s|}\right)e^{i\omega\vt_n\cdot\my_s}J_1(\omega|\mx-\my_s|)\right)\\
    &\times\left(\sum_{s'=1}^{M}\sum_{h=1}^{2}(\me_h\cdot\vt_n)\left(\frac{(\mx-\mz_{s'})\cdot\me_h}{|\mx-\mz_{s'}|}\right)e^{i\omega\vt_n\cdot\mz_{s'}}J_1(\omega|\mx-\mz_{s'}|)\right).
  \end{align*}

  First, applying (\ref{BesselRelation}), we can obtain
  \[\frac{1}{N}\sum_{n=1}^{N}\sum_{m=1}^{M}e^{i\omega\vt_n\cdot(\mx-\mz_m)}J_0(\omega|\mx-\mz_m|)=\sum_{m=1}^{M}J_0(\omega|\mx-\mz_m|)^2.\]  This leads us to
  \begin{equation}\label{term1}
    \frac{1}{N}\sum_{n=1}^{N}\Phi_1=\frac{1}{N}\sum_{n=1}^{N}\sum_{m=1}^{M}\bigg(e^{i\omega\vt_n\cdot(\mx-\mz_m)}+e^{-i\omega\vt_n\cdot(\mx-\mz_m)}\bigg)J_0(\omega|\mx-\mz_m|)=2\sum_{m=1}^{M}J_0(\omega|\mx-\mz_m|)^2
  \end{equation}
  and similarly to
  \begin{equation}\label{term2}
    \frac{1}{N}\sum_{n=1}^{N}\Phi_5=\frac{1}{N}\sum_{n=1}^{N}\sum_{s=1}^{S}\bigg(e^{i\omega\vt_n\cdot(\mx-\my_s)}+e^{-i\omega\vt_n\cdot(\mx-\my_s)}\bigg)J_0(\omega|\mx-\my_s|)=2\sum_{s=1}^{S}J_0(\omega|\mx-\my_s|)^2.
  \end{equation}

  Next, based on the orthonormal property of singular vectors, relations (\ref{Separated}) and (\ref{BesselRelation}), and the following asymptotic form
  \[J_0(\omega|\mz-\mz_{m'}|)\approx\sqrt{\frac{2}{\omega|\mz-\mz_{m'}|\pi}}\cos\left(\omega|\mz-\mz_{m'}|-\frac{\pi}{4}\right),\]
  we can derive
  \begin{align}
  \begin{aligned}\label{term3}
    \frac{1}{N}\sum_{n=1}^{N}\Phi_2&=-\frac{1}{N}\sum_{n=1}^{N}\left(\sum_{m=1}^{M}e^{i\omega\vt_n\cdot\mz_m}J_0(\omega|\mx-\mz_m|)\right)\left(\sum_{m'=1}^{M}e^{-i\omega\vt_n\cdot\mz_{m'}}J_0(\omega|\mx-\mz_{m'}|)\right)\\
    &=-\sum_{m=1}^{M}\sum_{m'=1}^{M}\left(\frac{1}{N}\sum_{n=1}^{N}e^{i\omega\vt_n\cdot(\mz-\mz_{m'})} J_0(\omega|\mx-\mz_m|)J_0(\omega|\mx-\mz_{m'}|)\right)\\
    &=-\sum_{m=1}^{M}\sum_{m'=1}^{M}J_0(\omega|\mz-\mz_{m'}|)J_0(\omega|\mx-\mz_m|)J_0(\omega|\mx-\mz_{m'}|)\\
    &=-\sum_{m=1}^{M}J_0(\omega|\mx-\mz_m|)^2.
  \end{aligned}
  \end{align}
  and similarly
  \begin{equation}\label{term4}
  \frac{1}{N}\sum_{n=1}^{N}\Phi_6=-\sum_{s=1}^{S}J_0(\omega|\mx-\my_s|)^2.
  \end{equation}

  For evaluating $\Phi_3$, let us perform an elementary calculus
  \begin{align*}
    \frac{1}{N}&\sum_{n=1}^{N}\left(i\sum_{m=1}^{M}\sum_{h=1}^{2}(\me_h\cdot\vt_n)e^{i\omega\vt_n\cdot(\mx-\mz_m)}\right)\left(\frac{(\mx-\mz_m)\cdot\me_h}{|\mx-\mz_m|}\right)J_1(\omega|\mx-\mz_m|)\\
    &=\sum_{m=1}^{M}\sum_{h=1}^{2}\left(i\frac{1}{N}\sum_{n=1}^{N}(\me_h\cdot\vt_n)e^{i\omega\vt_n\cdot(\mx-\mz_m)}\right)\left(\frac{(\mx-\mz_m)\cdot\me_h}{|\mx-\mz_m|}\right)J_1(\omega|\mx-\mz_m|)\\
    &=-\sum_{m=1}^{M}\sum_{h=1}^{2}\left(\frac{(\mx-\mz_m)\cdot\me_h}{|\mx-\mz_m|}\right)^2J_1(\omega|\mx-\mz_m|)^2.
  \end{align*}
  Then, we can conclude that
  \begin{equation}\label{term5}
    \frac{1}{N}\sum_{n=1}^{N}\Phi_3=2\sum_{m=1}^{M}\sum_{h=1}^{2}\left(\frac{(\mx-\mz_m)\cdot\me_h}{|\mx-\mz_m|}\right)^2J_1(\omega|\mx-\mz_m|)^2
  \end{equation}
  and
  \begin{equation}\label{term6}
    \frac{1}{N}\sum_{n=1}^{N}\Phi_7=2\sum_{s=1}^{S}\sum_{h=1}^{2}\left(\frac{(\mx-\my_s)\cdot\me_h}{|\mx-\my_s|}\right)^2J_1(\omega|\mx-\my_s|)^2.
  \end{equation}

  Finally, for $\Phi_4$, by applying following integral: for $\vt_n,\vt,\vx\in\mathbb{S}^1$, \[\frac{1}{N}\sum_{n=1}^{N}(\vt_n\cdot\vx)^2\approx\frac{1}{2\pi}\int_{\mathbb{S}^1}(\vt\cdot\vx)^2d\vt=\frac12,\]
  we can derive the following:
  \begin{align}
  \begin{aligned}\label{term7}
    \frac{1}{N}\sum_{n=1}^{N}\Phi_4=-&\frac{1}{N}\sum_{n=1}^{N}\left(\sum_{m=1}^{M}\sum_{h=1}^{2}(\me_h\cdot\vt_n)\left(\frac{(\mx-\mz_m)\cdot\me_h}{|\mx-\mz_m|}\right)e^{i\omega\vt_n\cdot\mz_m}J_1(\omega|\mx-\mz_m|)\right)\\
    &\times\left(\sum_{m'=1}^{M}\sum_{h''=1}^{2}(\me_{h''}\cdot\vt_n)\left(\frac{(\mx-\mz_{m'})\cdot\me_{h''}}{|\mx-\mz_{m'}|}\right)e^{i\omega\vt_n\cdot\mz_{m'}}J_1(\omega|\mx-\mz_{m'}|)\right)\\
    =-&\sum_{m=1}^{M}\left(\frac{1}{N}\sum_{n=1}^{N}\sum_{h=1}^{2}(\me_s\cdot\vt_n)^2\right)^2\sum_{h=1}^{2}\left\{\bigg(\frac{(\mx-\mz_m)\cdot\me_s}{|\mx-\mz_m|}\bigg)J_1(\omega|\mx-\mz_m|)\right\}^2\\
    =-&\sum_{m=1}^{M}\sum_{h=1}^{2}\left\{\bigg(\frac{(\mx-\mz_m)\cdot\me_h}{|\mx-\mz_m|}\bigg)J_1(\omega|\mx-\mz_m|)\right\}^2.
  \end{aligned}
  \end{align}
  Correspondingly,
  \begin{equation}\label{term8}
    \frac{1}{N}\sum_{n=1}^{N}\Phi_8=-\sum_{s=1}^{S}\sum_{h=1}^{2}\left\{\bigg(\frac{(\mx-\my_s)\cdot\me_h}{|\mx-\my_s|}\bigg)J_1(\omega|\mx-\my_s|)\right\}^2.
  \end{equation}
  Hence, by combining (\ref{term1})--(\ref{term8}), we can obtain the following mathematical structure
  \begin{multline*}
    |\mathbf{P}_{\mathrm{noise}}(\mf(\mr))|^2=1-\sum_{m=1}^{M}J_0(\omega|\mx-\mz_m|)^2-\sum_{m=1}^{M}\sum_{h=1}^2\bigg(\frac{(\mx-\mz_m)\cdot\me_h}{|\mx-\mz_m|}\bigg)^2J_1(\omega|\mx-\mz_m|)^2\\
  -\sum_{s=1}^{S}J_0(\omega|\mx-\my_s|)^2-\sum_{s=1}^{S}\sum_{h=1}^2\bigg(\frac{(\mx-\my_s)\cdot\me_h}{|\mx-\my_s|}\bigg)^2J_1(\omega|\mx-\my_s|)^2.
  \end{multline*}
  This enables us to obtain the desired result. This completes the proof.
\end{proof}

\begin{Rem}[Applicability of MUSIC]\label{Remark1}
Since $J_0(0)=1$, the value of $\mathcal{F}(\mx)$ will be sufficiently large when $\mx=\mz_m$ or $\my_s$ for all $m$ and $s$. Hence, based on the result in Theorem \ref{TheoremMUSIC}, the locations of $\Sigma_m$ and $\Delta_s$ can be identified via the map of $\mathcal{F}(\mz)$. This is the reason why it is possible to detect the locations of small inhomogeneities as well as random scatterers. Note that for a successful detection, based on the hypothesis in Theorem \ref{TheoremMUSIC}, the value of $N$ (at least, greater than $3M+3S$) and $\omega$ must be sufficiently large enough. If applied frequency is low or total number of $N$ is small, poor result would appear in the map of $\mathcal{F}(\mx)$.
\end{Rem}

\begin{Rem}[Discrimination of singular values]\label{Remark2}
  Theoretically, if the size, permittivity, and permeability of the random scatterers are smaller than those of the inhomogeneities, then $\sigma_s<\sigma_m$ for all $m$ and $s$. This means that if it were possible to discriminate singular values associated with small inhomogeneities then, the structure of $\mathcal{F}(\mz)$ would become
  \[\mathcal{F}(\mx)\approx\left(1-\sum_{m=1}^{M}J_0(\omega|\mx-\mz_m|)-\sum_{m=1}^{M}\sum_{h=1}^2\bigg(\frac{(\mx-\mz_m)\cdot\me_h}{|\mx-\mz_m|}\bigg)^2J_1(\omega|\mx-\mz_m|)^2\right)^{-1/2}.\]
  Hence, it is expected that more good results can be obtained. Our approach presents an improvement. However, if the relation $\sigma_s<\sigma_m$ were no longer valid, the locations of random scatterers would have to be identified via MUSIC such that poor results would appear in the map of $\mathcal{F}(\mx)$.
\end{Rem}

\section{Results of numerical simulations}\label{sec:5}
Selected results of numerical simulations are presented here to support the identified structure of the MUSIC-type imaging function. In this section, we only consider the dielectric permittivity contrast case, i.e., we set $\eps_m=3$, $\eps_0=1$, and $\mu_m=\mu_s=\mu_0$ for all $m$ and $s$. The radius of all $\Sigma_m$ and $\Delta_s$ are set to $0.1$ and $0.05$, respectively. The applied angular frequency is $\omega=2\pi/\lambda$ and a total of $N$ number of incident directions are applied such that
\[\vt_j=-\left[\cos\frac{2\pi(j-1)}{N},\frac{2\pi(j-1)}{N}\right]^T,\quad j=1,2,\cdots,N.\]

$M=3$ small inhomogeneities are selected with locations $\mz_1=[0.25,0]^T$, $\mz_2=[-0.4,0.5]^T$, and $\mz_3=[-0.3,-0.7]^T$. We set $S=100$ number of small scatterers as being randomly distributed in $\Omega=[-1,1]\times[-1,1]\subset\mathbb{R}^2$ such that
\[\my_s=[\eta_1(-1,1),\eta_2(-1,1)]^T\]
for all $s$ and also select the permittivities randomly as
\[\eps_s=\eta_3(1,2),\]
where $\eta_p(a,b)$, and $p=1,2,$ and $3,$ is an arbitrary real value within $[a,b]$. Refer to Fig. \ref{Distibution} for a sketch of the distribution of the three inhomogeneities and random scatterers.

\begin{figure}[h]
\centering
\includegraphics[width=0.5\textwidth]{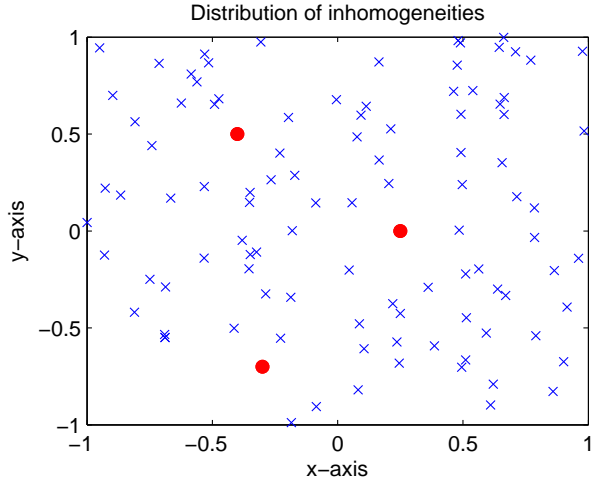}
\caption{\label{Distibution}Distribution of inhomogeneities (red-colored dots) and random scatterers (blue-colored `$\times$' mark).}
\end{figure}

The far-field elements of MSR matrix $\mathbb{K}$ is generated by means of the Foldy-Lax framework   to avoid an \textit{inverse crime}. After the generation, a singular value decomposition of $\mathbb{K}$ is performed via the MATLAB command \texttt{svd}. The nonzero singular values of $\mathbb{K}$ are discriminated as follows: first, a $0.1-$threshold scheme (by first choosing the $j$ singular values $\sigma_j$ such that $\frac{\sigma_j}{\sigma_1}\geq0.1$) is applied based on \cite{PL3} and second, the first $3-$singular values are selected.

Fig. \ref{Result} exhibits the distribution of the normalized singular values of $\mathbb{K}$ and maps of $\mathcal{F}(\mx)$ with the $0.1-$threshold scheme and with selection of the first $3-$singular values when $\lambda=0.3$ and $N=32$. Note that due to the huge number of artifacts it is very hard to identify the locations of $\Sigma_m$ with the $0.1-$threshold scheme but, fortunately in this example, one can discriminate three nonzero singular values such that, based on the following Remark \ref{Remark2}, the locations of $\Sigma_m$ can be identified more clearly. This result supports the derived mathematical structure in Theorem \ref{TheoremMUSIC}.

\begin{figure}[h]
\centering
\includegraphics[width=0.49\textwidth]{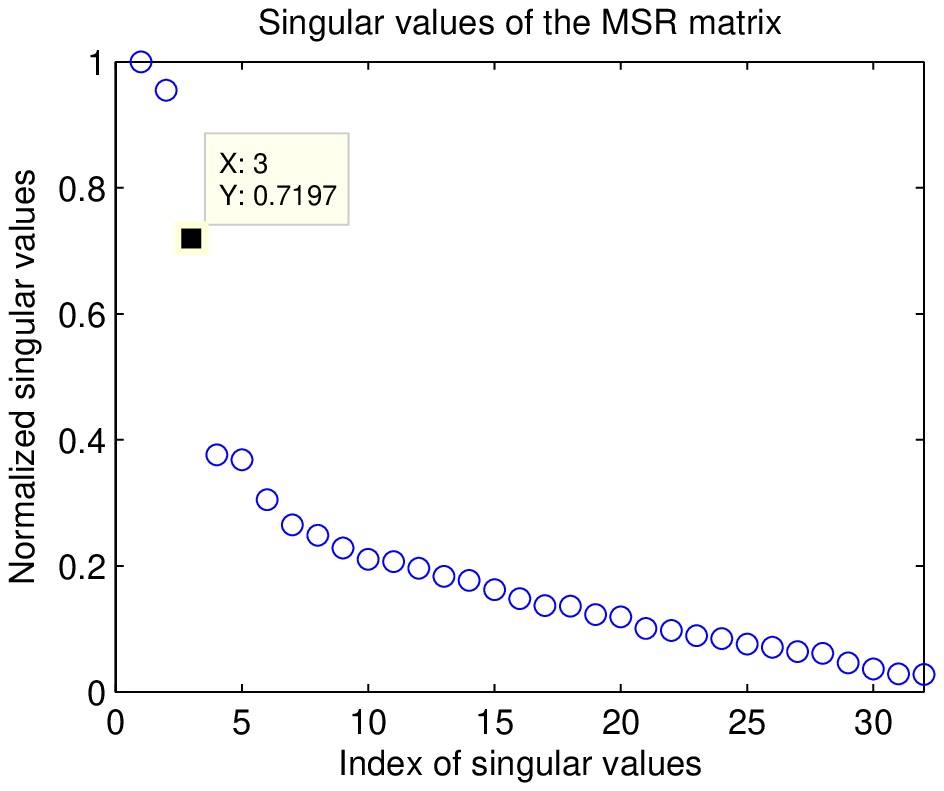}
\includegraphics[width=0.49\textwidth]{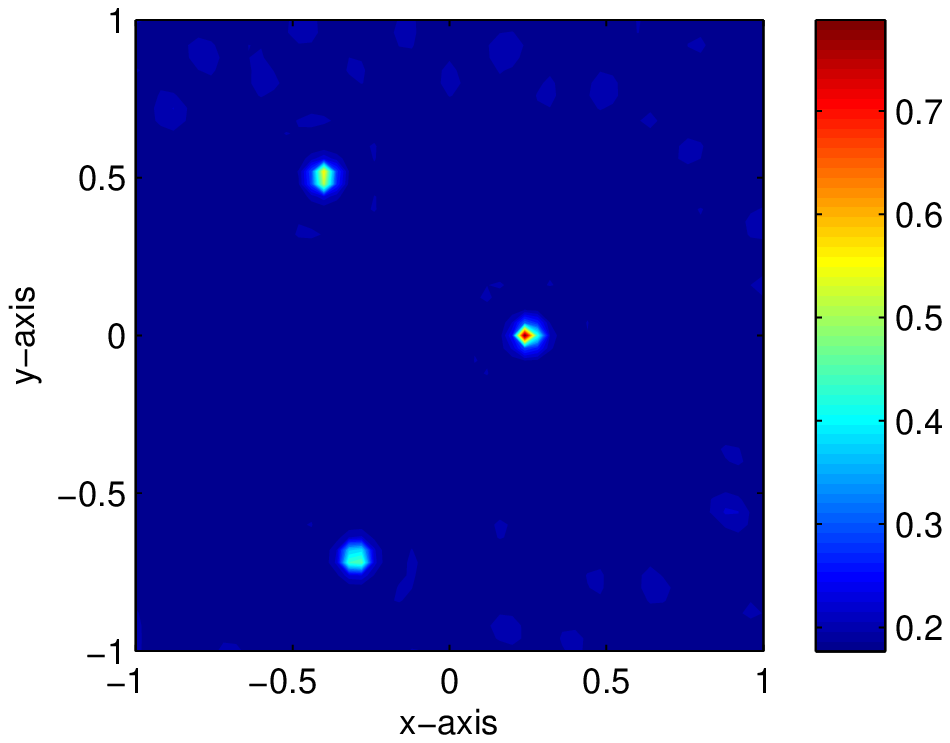}\\
\includegraphics[width=0.49\textwidth]{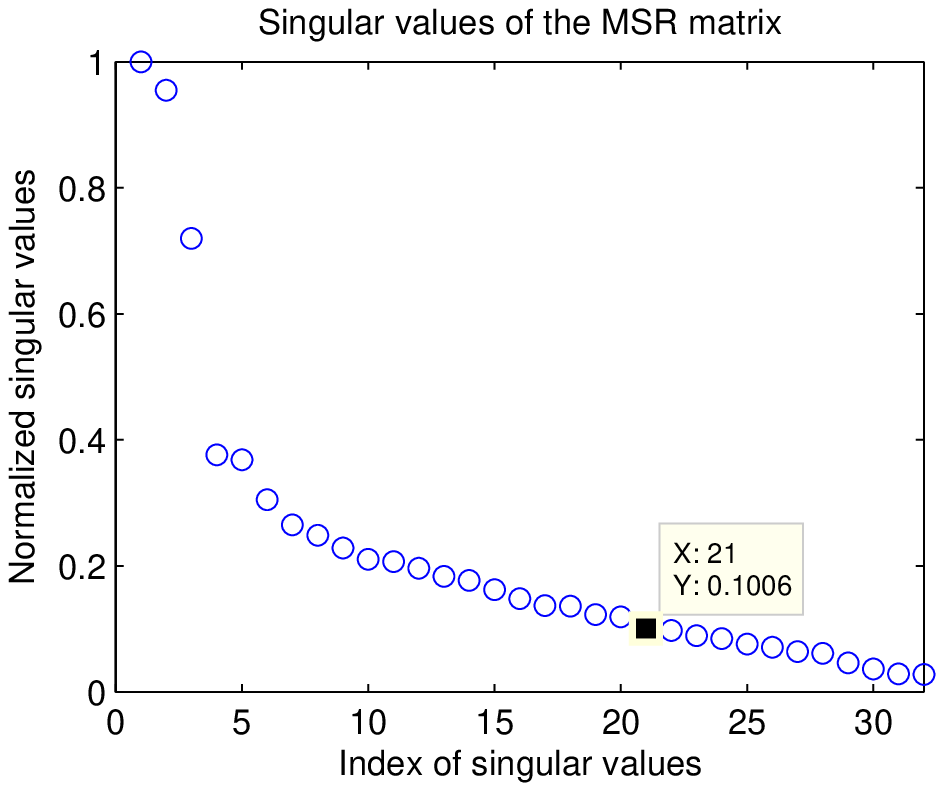}
\includegraphics[width=0.49\textwidth]{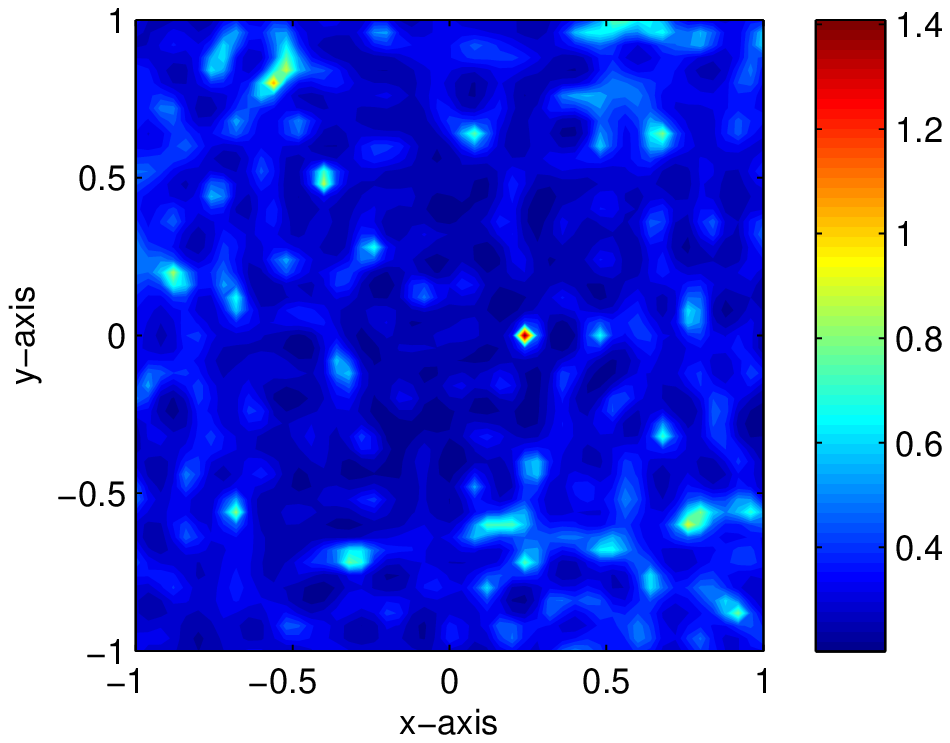}
\caption{Distribution of normalized singular values (left column) and maps of $\mathcal{F}(\mx)$ with first $3-$singular values (top, right) and with $0.1-$threshold scheme (bottom, right).}\label{Result}
\end{figure}

Now, let us examine the effect of total number of directions $N$ in the extreme cases. Figure \ref{Result-small} exhibits normalized singular values and map of $\mathcal{F}(\mx)$ with small number of $N=5$ when $\lambda=0.4$. Based on Remark \ref{Remark1}, the value of $N$ must be sufficiently large so, as we expected, locations of $\Sigma_m$ cannot be identified via the map of $\mathcal{F}(\mx)$ with small $N$.

\begin{figure}[h]
\centering
\includegraphics[width=0.49\textwidth]{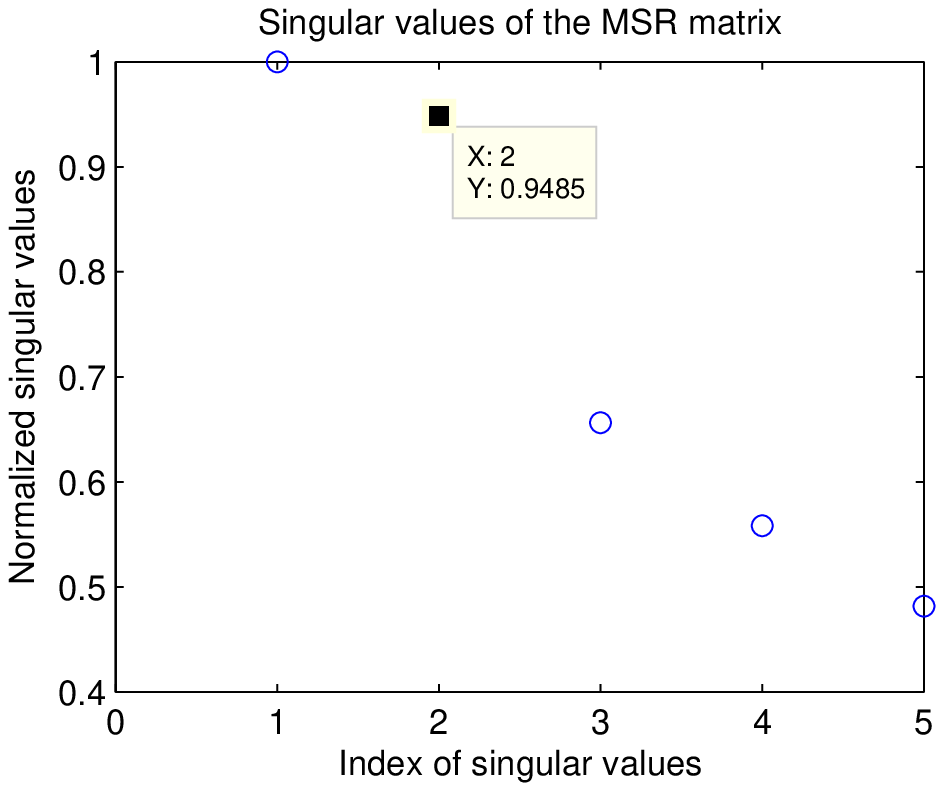}
\includegraphics[width=0.49\textwidth]{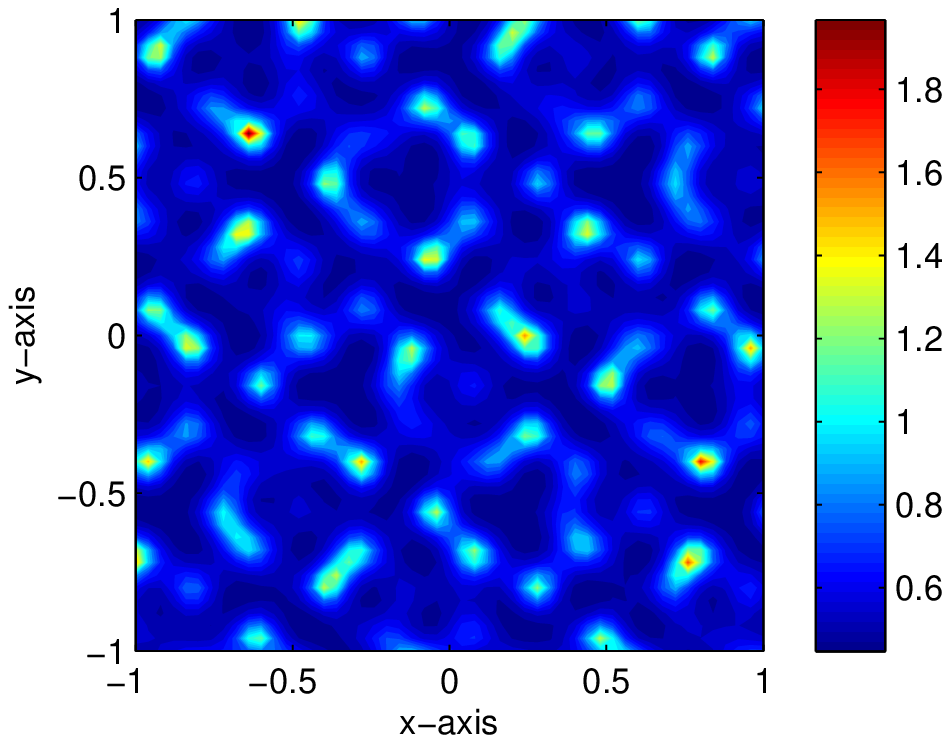}
\caption{Distribution of normalized singular values (left) and map of $\mathcal{F}(\mx)$ with first $2-$singular values (right).}\label{Result-small}
\end{figure}

Opposite to the previous result, Figure \ref{Result-large} displays normalized singular values and maps of $\mathcal{F}(\mx)$ with large number of $N=256$ when $\lambda=0.4$. Similar to the results in Figure \ref{Result}, locations of $\Sigma_m$ can be examined clearly via the selection of first $3-$singular values. Applying $0.1-$threshold, it is very hard to identify locations of $\Sigma_m$ but, opposite to the result in Figure \ref{Result}, their locations can be recognized even though some artifacts are still exist.

\begin{figure}[h]
\centering
\includegraphics[width=0.49\textwidth]{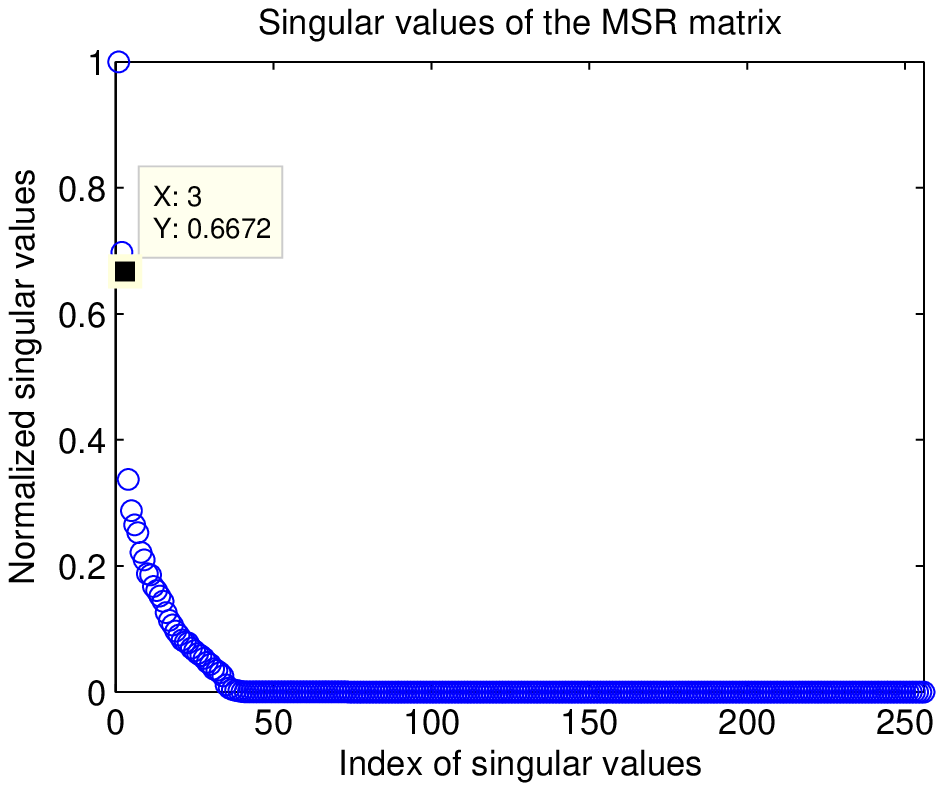}
\includegraphics[width=0.49\textwidth]{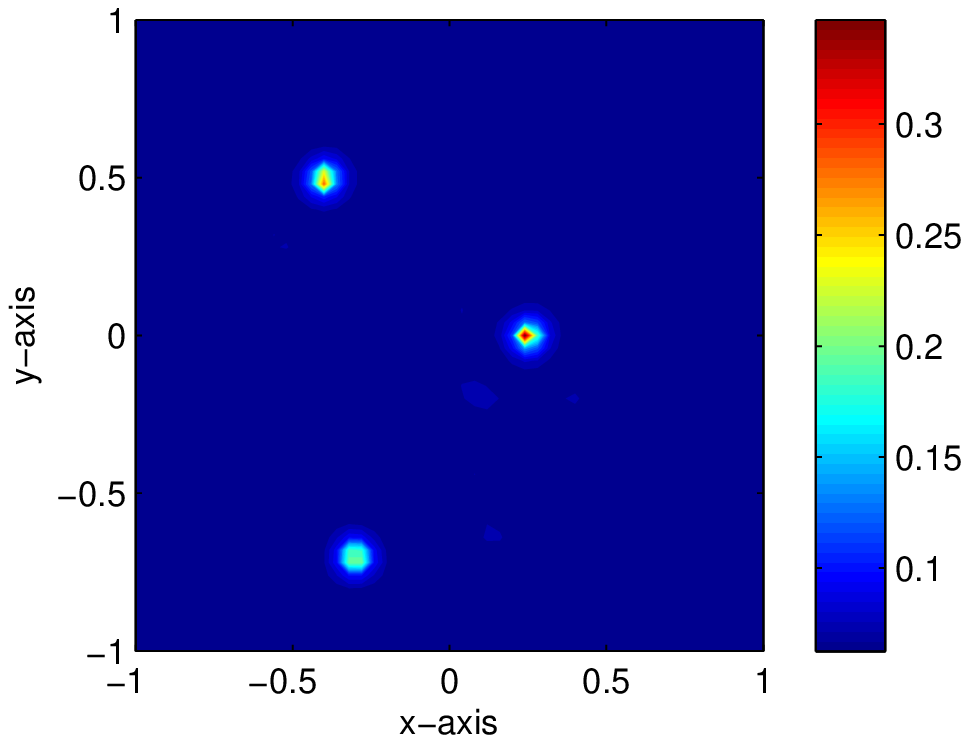}\\
\includegraphics[width=0.49\textwidth]{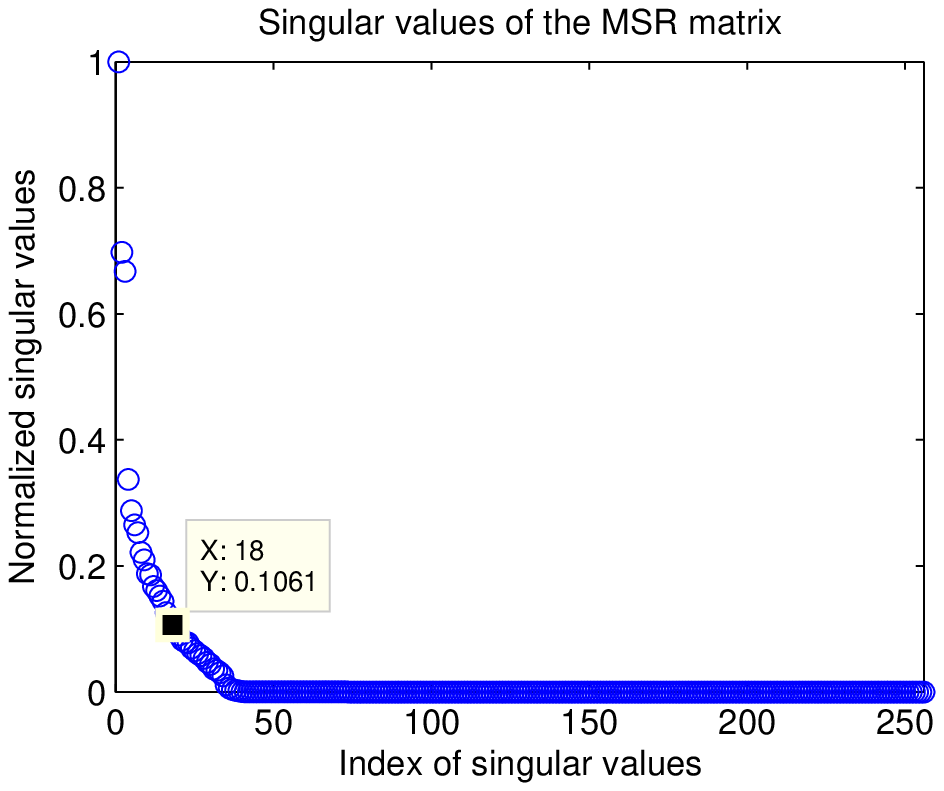}
\includegraphics[width=0.49\textwidth]{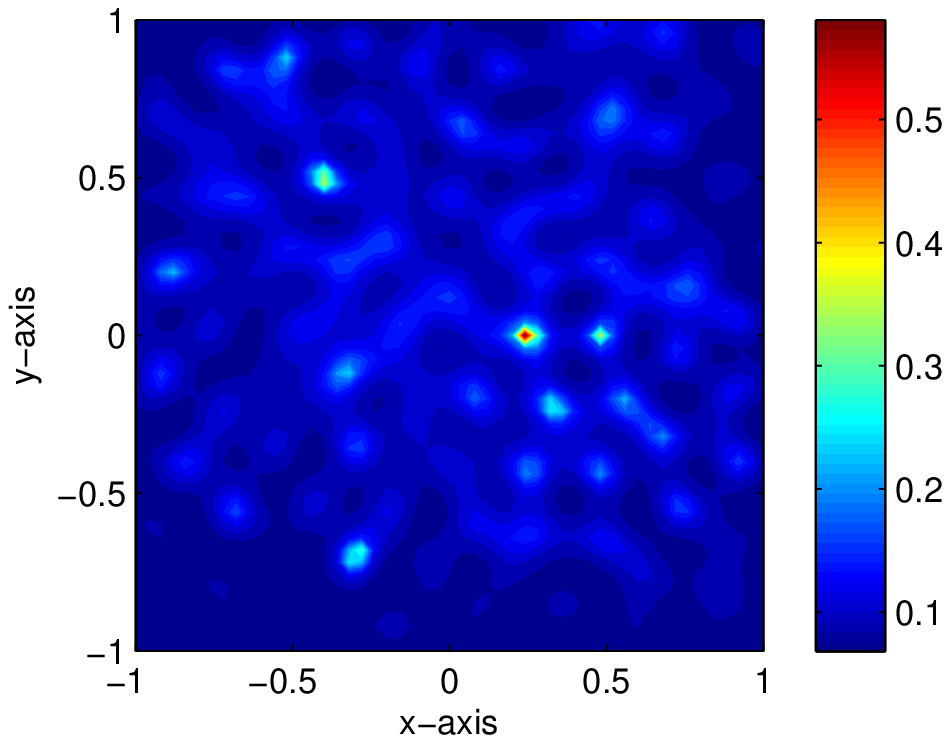}
\caption{Distribution of normalized singular values (left column) and maps of $\mathcal{F}(\mx)$ with first $3-$singular values (top, right) and with $0.1-$threshold scheme (bottom, right).}\label{Result-large}
\end{figure}

On the basis of recent works \cite{AGKPS,HSZ1}, it has been confirmed that MUSIC is robust with respect to the random noise. In order to examine the robustness, assume that $10$ dB Gaussian random noise is added to the unperturbed data $u_{\mathrm{far}}(\vv_j,\vt_l)$. Throughout results in Figure \ref{Result-noise} when $N=32$ and $\lambda=0.3$, although some blurring appears in the map of $\mathcal{F}(\mx)$, we can easily find proper singular values and obtain an accurate image. It is interesting to observe that opposite to the results in Figure \ref{Result}, locations of $\Sigma_m$ can be detected even though existence of some artifacts.

\begin{figure}[h]
\centering
\includegraphics[width=0.49\textwidth]{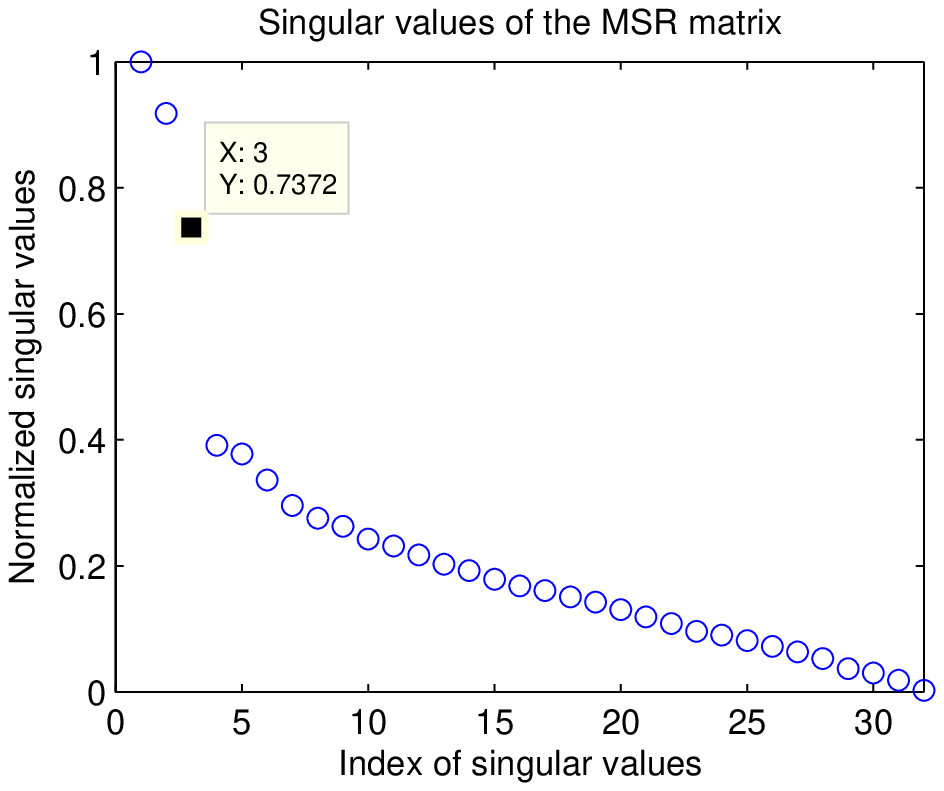}
\includegraphics[width=0.49\textwidth]{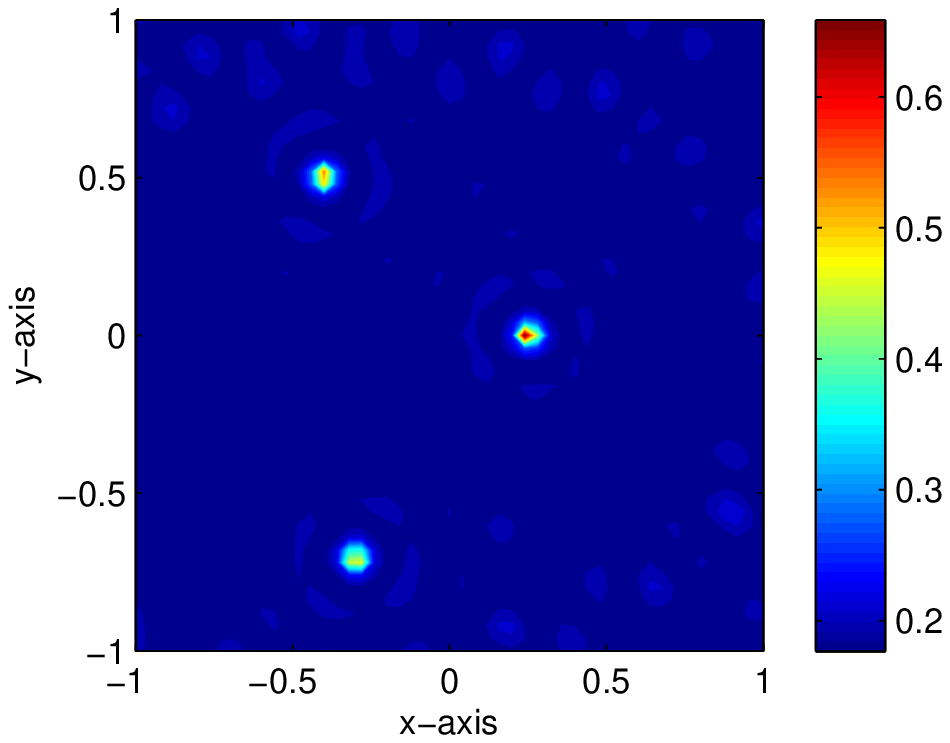}\\
\includegraphics[width=0.49\textwidth]{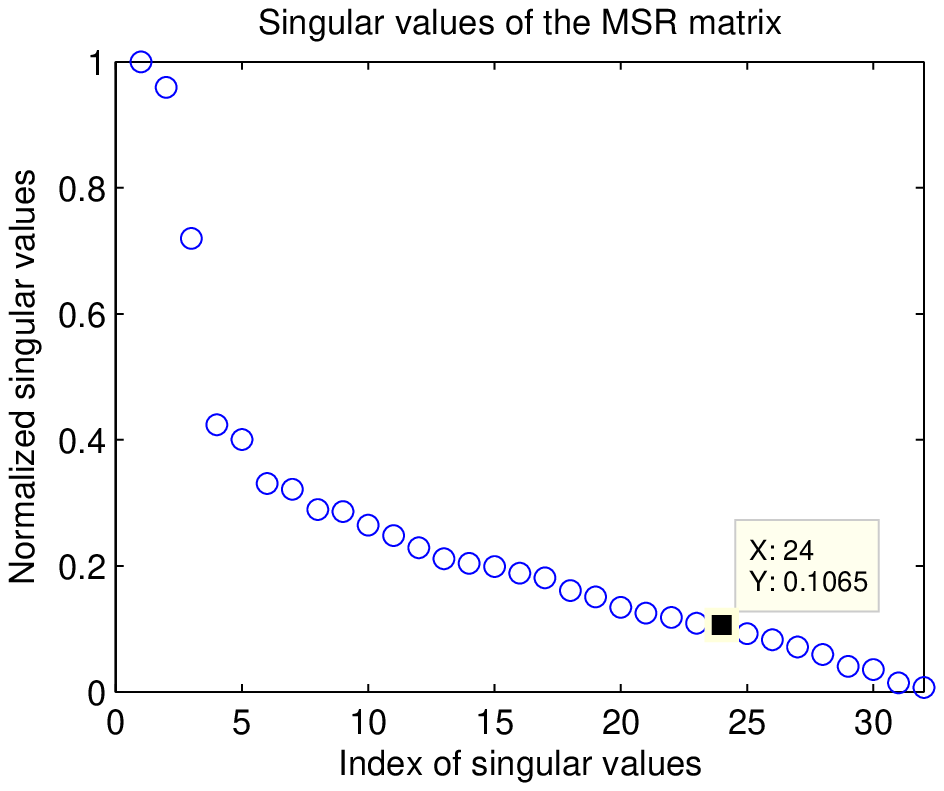}
\includegraphics[width=0.49\textwidth]{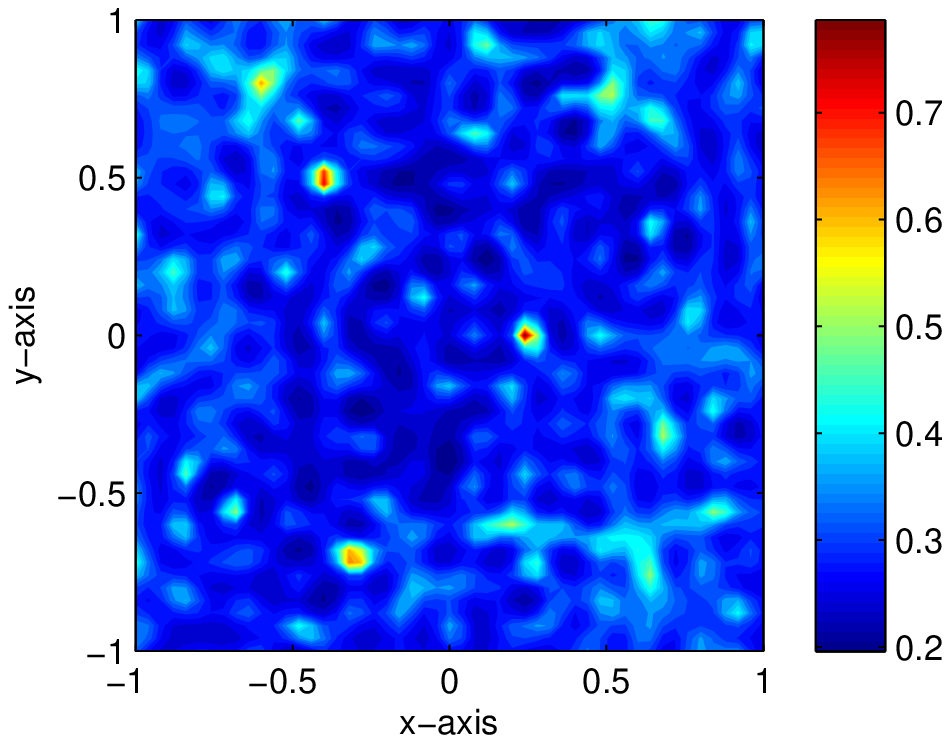}
\caption{Distribution of normalized singular values (left column) and maps of $\mathcal{F}(\mx)$ with first $3-$singular values (top, right) and with $0.1-$threshold scheme (bottom, right) when $N=32$, $\lambda=0.3$, and collected far-field data is perturbed by a white Gaussian random noise.}\label{Result-noise}
\end{figure}

From the above results, we can examine that by having small perturbations of random scatterers $\Delta_s$, their effects to the scattered fields are quite small so that $\Sigma_m$ can be discriminated very accurately. Opposite to the this examination, let us consider the effect of $\Delta_s$ when their size and permittivities satisfy $r_s=0.1$ and $\eps_s=\eta(2.5,3)$, respectively (remember that $r_m=0.1$ and $\eps_m\equiv3$ for all $m$). In this example, it is very hard to discriminate nonzero singular values associated with $\Sigma_m$ so that it is impossible to detect their exact locations, refer to Figure \ref{Result-distinguish} when $N=32$ and $\lambda=0.4$.

\begin{figure}[h]
\centering
\includegraphics[width=0.49\textwidth]{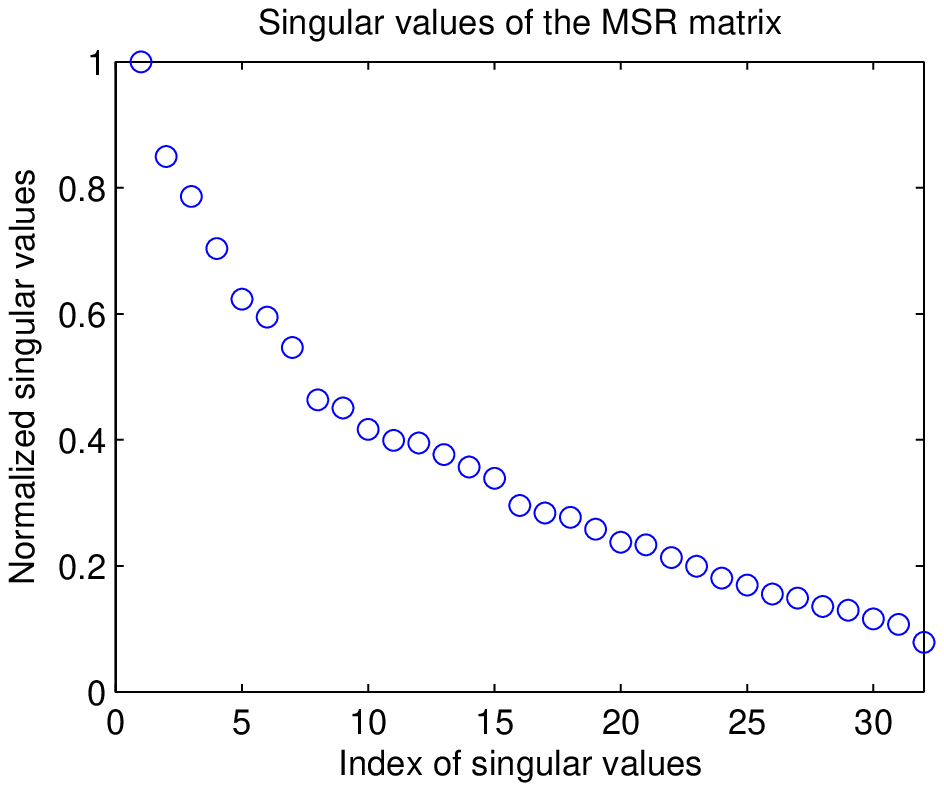}
\includegraphics[width=0.49\textwidth]{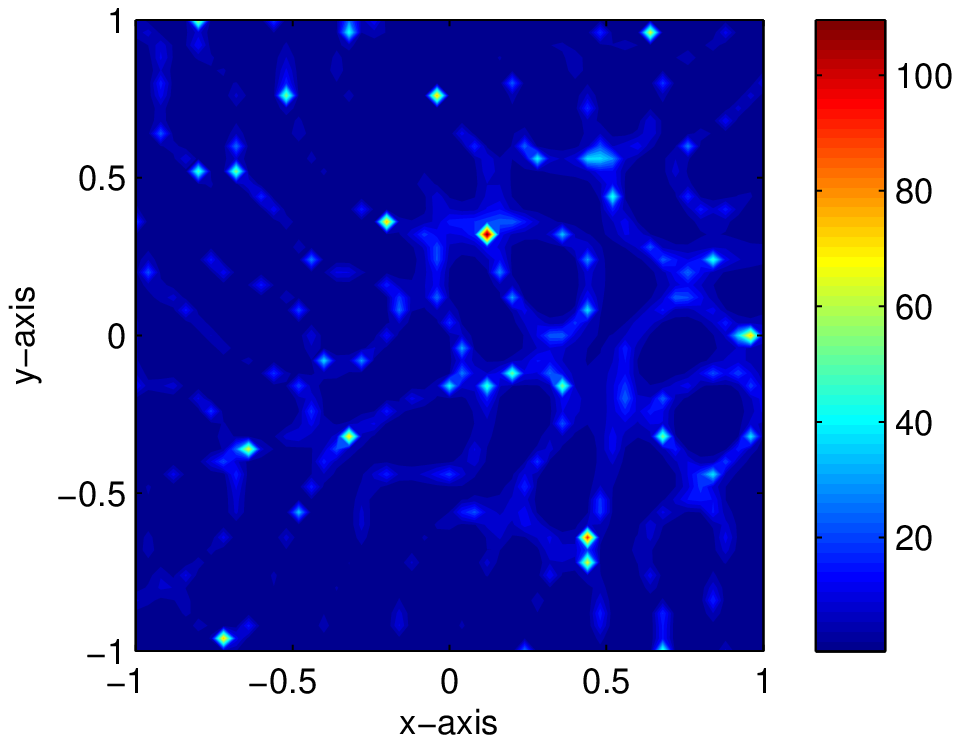}
\caption{Distribution of normalized singular values (left) and map of $\mathcal{F}(\mx)$ with $0.1-$threshold scheme (right).}\label{Result-distinguish}
\end{figure}

It is well-known that using multi-frequency improves the imaging performance, refer to \cite{AGKPS,P-SUB3,P-SUB1,JP2}. At this moment, we consider multi-frequency MUSIC-type imaging in order to compare the imaging performance against the traditional single-frequency one. For given $F-$ different frequencies $0<\omega_1<\omega_2<\cdots<\omega_F$, SVD of MSR matrix $\mathbb{K}(\omega_f)$ is
\[\mathbb{K}(\omega_f)\approx\sum_{m=1}^{3M}\sigma_m(\omega_f)\mathbf{U}_m(\omega_f)\mathbf{V}_m^*(\omega_f)+\sum_{s=3M+1}^{3M+3S}\sigma_s(\omega_f)\mathbf{U}_s(\omega_f)\mathbf{V}_s^*(\omega_f).\]
Then, by choosing test vector
\[\mf(\mx;\omega_f)=\frac{1}{\sqrt{N}}\bigg[e^{i\omega_f\vt_1\cdot\mx},e^{i\omega_f\vt_2\cdot\mx},\cdots,e^{i\omega_f\vt_N\cdot\mx}\bigg]^T,\]
we can survey the projection operator onto the null (or noise) subspace such that
\[\mathbf{P}_{\mathrm{noise}}(\mf(\mx;\omega_f)):=\left(\mathbb{I}_N-\sum_{m=1}^{3M+3S}\mathbf{U}_m(\omega_f)\mathbf{U}_m^*(\omega_f)\right)\mf(\mx;\omega_f),\]
and correspondingly multi-frequency MUSIC-type imaging function $\mathcal{Q}(\mx;F)$ can be introduced as
\[\mathcal{Q}(\mx;F)=\left|\frac{1}{F}\sum_{f=1}^{F}\mathbf{P}_{\mathrm{noise}}(\mf(\mx;\omega_f))\right|^{-1}.
\]

Figure \ref{Result-frequency} shows maps of $\mathcal{Q}(\mx;10)$, where $\omega_f=2\pi/\lambda_s$. Here, $N=32$ directions are applied and $\lambda_f$ are equi-distributed in the interval $[\lambda_F,\lambda_1]$ with $\lambda_1=0.7$ and $\lambda_F=0.3$. By comparing results in Figure \ref{Result}, we can observe that unexpected artifacts have been eliminated so that applying multiple frequencies yields a more accurate result then single frequency.

\begin{figure}[h]
\centering
\includegraphics[width=0.49\textwidth]{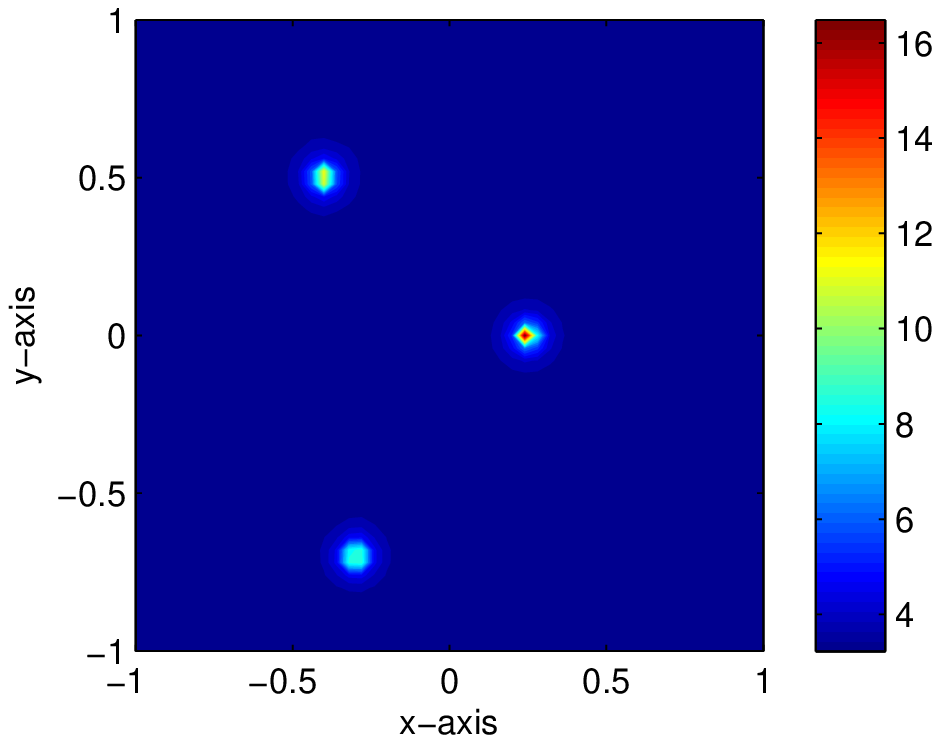}
\includegraphics[width=0.49\textwidth]{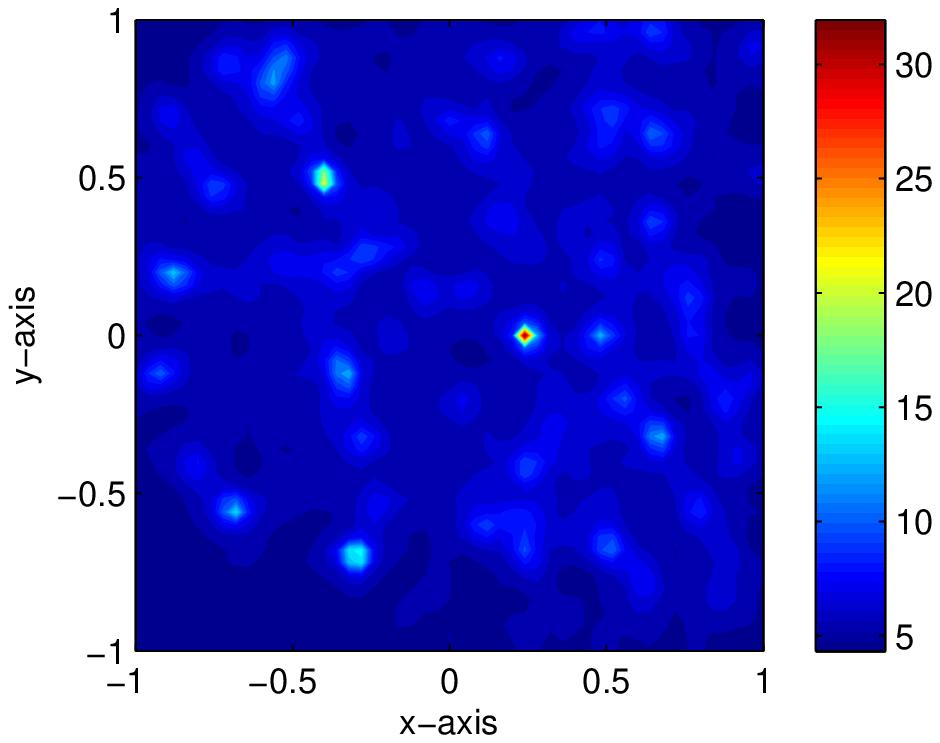}
\caption{Maps of $\mathcal{Q}(\mx;10)$ with first $3-$singular values (left) and with $0.1-$threshold scheme (right).}\label{Result-frequency}
\end{figure}

\section{Concluding remarks}\label{sec:6}
The mathematical structure of a MUSIC-type imaging function is carefully identified by establishing a relationship with integer ordered Bessel functions. This is based on the fact that the elements of the MSR matrix can be expressed by an asymptotic expansion formula. The identified structure explains some unexplained phenomena and provides a method for improvements.

Based on recent work \cite{AILP}, the electric field $\mE$ in the existence of small inhomogeneity with radius $r$ can be expressed as follows:
\begin{multline*}
\mE(\mx)=\mE_0(\mx)+r^3|\mB_m|\sum_{m=1}^{M}\bigg(k^2\frac{3(\eps_m-\eps_0)}{\eps_m+2\eps_0}\mG(\mx,\mz_m)\cdot\mE_0(\mz_m)\\
-i\omega\mu_0\frac{3(\mu_m-\mu_0)}{\mu_m+2\mu_0}\nabla\times\mG(\mx,\mz_m)\cdot\mH_0(\mz_m)+\bigg)+O(r^4),
\end{multline*}
where electromagnetic fields $(\mE_0,\mH_0)$ are the solutions of the Maxwell equations
\[\left\{\begin{array}{l}
\medskip\nabla\times\mE_0=i\omega\mu_0\mH_0\quad\mbox{in}\quad\mathbb{R}^3\\
\medskip\nabla\times\mH_0=-i\omega\eps_0\mE_0+\mathbf{J}_0\quad\mbox{in}\quad\mathbb{R}^3\\
\medskip\displaystyle\lim_{|\mr|\to\infty}\mr\bigg(\nabla\times\mE_0-ik\frac{\mr}{|\mr|}\times\mE_0\bigg)=0\\
\medskip\displaystyle\lim_{|\mr|\to\infty}\mr\bigg(\nabla\times\mH_0-ik\frac{\mr}{|\mr|}\times\mH_0\bigg)=0
\end{array}\right.\]
and $\mG$ is Green's function
\[\mG(\mx,\mz_m):=\left(\left[\begin{array}{ccc}1&0&0\\0&1&0\\0&0&1\end{array}\right]+\frac{\nabla\nabla}{k^2}\right)\frac{e^{ik|\mx-\mz_m|}}{4\pi|\mx-\mz_m|}.\]
Thus, by applying above asymptotic expansion formula and through the similar process in Theorem \ref{TheoremMUSIC}, the result in this paper can be extended to the three dimension problem so that MUSIC will be applicable for detecting three-dimensional inhomogeneities surrounded by random scatterers.

In comparison with the MUSIC, other closely related reconstruction algorithms such as linear sampling method \cite{CHM,HM,KR}, subspace migration \cite{P-SUB3,JP1,P-SUB1}, and direct sampling method \cite{LLZ,IJZ1,IJZ2} will be applicable for detecting inhomogeneities in random medium. Analysis of imaging functions and exploring their certain properties will be the forthcoming work.

\bibliographystyle{unsrt}
\bibliography{References}
\end{document}